\documentclass[a4paper,draft]{ZbRad}
\usepackage{amssymb}
\usepackage{amscd} 

\theoremstyle{plain}
 \newtheorem{thm}{Theorem}[section]
 \newtheorem{prop}{Proposition}[section]
 
 \newtheorem{cor}{Corollary}[section]
\theoremstyle{definition}

 \newtheorem{dfn}{Definition}[section]
\numberwithin{equation}{section}

\renewcommand{\leq}{\leqslant}
\renewcommand{\geq}{\geqslant}
\renewcommand{\setminus}{\smallsetminus}
\newfont{\TITf}{cmssdc10 scaled 1440}

\newcommand{\To}{\longrightarrow}

\title[An introduction to multiple gaps]{An introduction to multiple gaps}

\author[Antonio Avil\'{e}s]{Antonio Avil\'{e}s}

\begin{document}

\setcounter{page}{1}

\vspace*{40mm}

\thispagestyle{empty}

{
\TITf\setlength{\parskip}{\smallskipamount}

\begin{center}
Antonio Avil\'{e}s

\bigskip\bigskip\bigskip

{AN INTRODUCTION TO MULTIPLE GAPS}

\end{center}
}

\vspace{6em}{\leftskip3em\rightskip3em
\emph{Abstract}.
This is an introductory article to the theory of multiple gaps.



}

\newpage\thispagestyle{empty}

\maketitle
\tableofcontents

The aim of this article is to provide an introduction to the theory of multiple gaps, recently developed by Stevo Todorcevic and the author in a series of papers \cite{multiplegaps, stronggaps, IHES}, and in some unpublished material that can be found in \cite{analyticmultigaps}. This is not intended to be an exhaustive survey, we have just focused on a particular topic, mostly concerning analytic multiple gaps, and present some selected results that are often simplified, rather than in general form. We neither tried to give all the proofs, but just some arguments and ideas. The purpose is to have a reasonably light and easy to read note, from which one can get an idea of what the results and the techiques are, and serve as a motivation for the more interested reader to look for further details.\\

We also tried to assume as little background as possible, so that this becomes a suitable reading for non-experts and students, where they may get some basic ideas of topics like descriptive set-theory, Ramsey theory or Banach spaces of continuous functions, together with references where to learn about it.\\

A multiple gap is nothing else than finitely many families of subsets of a countable set, with some conditions about how these families are \emph{mixed}. This is a fairly elementary kind of object that could arise in a variety of contexts. The formal definition and some elementary and historical comments are found in Section~\ref{sectiondefinitions}. Although there are some brief considerations about cardinal $\aleph_1$ in that section, most of the paper is devoted to nicely definable objects: the analytic multiple gaps. In Section~\ref{sectioncritical}, which collects results from~\cite{multiplegaps, IHES}, we present some examples of gaps that we call \emph{critical} and we explain the unique importance of these particular examples in the general theory. Taking these critical gaps as a starting point, the theory is enriched by the intervention of Ramsey theory and what we call the \emph{first-move and record combinatorics} of the $n$-adic tree, explained in Sections~\ref{sectionstrong} and~\ref{sectionweak}, that present results from \cite{stronggaps} and \cite{IHES,analyticmultigaps} respectively. The rest of sections present applications of the theory in different settings, like the topology of $\beta\omega\setminus\omega$, sequences of vectors in Banach spaces and operators in $\ell_\infty/c_0$. Sections~\ref{sectioncountablyseparated} and \ref{sectionCK} present results from \cite{multiplegaps}, while the rest of sections explain ideas from~\cite{analyticmultigaps}.\\

Before getting into matter, just a warning about some notations. The set of natural numbers is denoted as $\omega = \{0,1,2,\ldots\}$, and each natural number is identified with its set of predecessors $n=\{0,1,2,\ldots,n-1\}$. In this way, $$\{x_n\}_{n\in \omega} = \{x_n\}_{n<\omega} = \{x_n : n<\omega\} =  \{x_0,x_1,x_2,\ldots\}$$
are alternative ways of denoting an infinite sequence, and
$$\{x_i\}_{i\in n} = \{x_i\}_{i<n} = \{x_i : i<n\} =  \{x_0,x_1,x_2,\ldots,x_{n-1}\}$$
are alternative ways of denoting a finite sequence of $n$ elements. Notice the set-theoretic tradition of starting counting from 0.\\

\section{General definitions}\label{sectiondefinitions}

Let $N$ be a countable set. We are interested in comparing subsets of $N$ modulo finite sets, so for $a,b\subset N$, we write $a\subset^\ast b$ if $a\setminus b = \{i\in a : i\not\in b\}$ is finite, that is, if all but finitely many elements of $a$ belong to $b$. We write $a=^\ast b$ if $a\subset^\ast b$ and $b\subset^\ast a$.

\begin{dfn}
Let $\{\Gamma_i : i<n\}$ be a finite collection of families of subsets of $N$. We say that these families are separated if we can find $a_0,\ldots,a_{n-1}$ subsets of $N$ such that
\begin{enumerate}
\item $x\subset^\ast a_i$ for each $x\in\Gamma_i$, 
\item $\bigcap_{i<n}a_i =^\ast \emptyset$.
\end{enumerate} 
\end{dfn}

Allowing finite errors is important here, otherwise we would be just saying that $\bigcap_{i<n}\bigcup\Gamma_i =\emptyset$.

\begin{dfn}
A finite family $\Gamma = \{\Gamma_i : i<n\}$ of families of subsets of $N$ is said to be an $n_\ast$-gap if
\begin{enumerate}
\item If we pick $x_i\in\Gamma_i$, then $\bigcap_{i<n}x_i = ^\ast\emptyset$ 
\item The families $\{\Gamma_i : i<n\}$ cannot be separated
\item Each family $\Gamma_i$ is infinitely hereditary: if $x\in\Gamma_i$ and $y$ is an infinite subset of $x$, then $y\in\Gamma_i$
\end{enumerate}
\end{dfn}

The tricky point is that, since we are allowing finite errors, taking $a_i = \bigcup \Gamma_i$ will not work to separate the families. The last requirement in the definition is just a harmless technical condition that will simplify the language for the kind of problems that we are going to discuss. Sometimes we may present examples of $n_\ast$-gaps where the families $\Gamma_i$ may not be hereditary; in that case, it is implicitly assumed that we close $\Gamma_i$ under subsets in order to fit in the above definition.\\

Two families of sets $I$ and $J$ are orthogonal if $x\cap y = ^\ast \emptyset$ whenever $x\in I$ and $y\in J$.   

\begin{dfn}
An $n$-gap is an $n_\ast$-gap $\Gamma$ in which moreover the families $\{\Gamma_i : i<n\}$ are pairwise orthogonal\footnote{The original definition of $n$-gap given in \cite{multiplegaps} is different from this one. It is given in an abstract Boolean algebra $\mathcal{B}$ -a degree of abstraction that we are not interested in here, we would only consider $\mathcal{B} =\mathcal{P}(\omega)/fin$- and required the families $\Gamma_i$ to be ideals (closed under finite unions) -a restriction that we do not want to make here-.}.
\end{dfn}

Although the general notion of $n_\ast$-gap may be useful at some points, we shall be concerned mainly with $n$-gaps. At this point, the sensible reader wants to see examples of $n$-gaps, to get an idea of how possibly one can get families which are on the one hand orthogonal but on the other hand cannot be separated. These examples are presented in Section~\ref{sectioncritical}, and the impatient reader can go straight to them. In the rest of this section we shall recall some basic facts and we shall discuss how the theory of $n$-gaps relates to the classical theory of gaps.\\

A 2-gap is what is traditionally called just a \emph{gap} by set-theorists. The above definition is the multidimensional generalization of it that was introduced in \cite{multiplegaps}, but 2-gaps are much older, going back to Hausdorff~\cite{Ha2}. First of all, why the name \emph{gap}? Well, if $\{\Gamma_0,\Gamma_1\}$ is a 2-gap according to the above definition, we can consider the family $\Gamma_1^\ast = \{N\setminus x: x\in\Gamma_1\}$. In that case, we have that $x\subset^\ast y$ for every $x\in\Gamma_0$ and $y\in\Gamma_1^\ast$ (because of mutual orthogonality) but there is no element $a$ such that $x\subset^\ast a \subset^\ast y$ for all $x\in\Gamma_0$, $y\in\Gamma_1^\ast$ (becase of non-separation). There is nothing in between, so there is a gap.   

The first observation is that we cannot produce gaps out of countable families:

\begin{prop}\label{separationofcountable}
If $I$ and $J$ are two orthogonal countable families of subsets of $N$, then they are separated.
\end{prop}

\begin{proof}
Say that $I = \{a_n : n<\omega\}$ and $J=\{b_n : n<\omega\}$. By the considering the finite unions $\tilde{a}_n = \bigcup_{k<n}a_k$ and $\tilde{b}_n = \bigcup_{k<n}b_k$, we can suppose without loss of generality that  $a_1\subset a_2 \subset$ and $b_1\subset b_2\subset\cdots$. By orthogonality $a_i\cap b_j$ is finite for all $i,j$. Inductively on $n$, we can make finite modifications $a'_n$ and $b'_n$ of each $a_n$ and $b_n$, so that still $a'_1\subset a'_2 \subset$ and $b'_1\subset b'_2\subset\cdots$ and moreover $a'_n\cap b'_n = \emptyset$. Finally, the sets $a = \bigcup_n a'_n$ and $b=\bigcup_n b'_n$ separate $I$ and $J$.
\end{proof}

Going beyond countable cardinality, Hausdorff produced his famous gap:

\begin{thm}[Hausdorff]\label{Hausdorff} There exist two orthogonal families of subsets of $\omega$ of cardinality $\aleph_1$ which are not separated.
\end{thm}

The two families are a special kind of $\omega_1$-chains in the order $\subset^\ast$. An exposition of this can be found in \cite{FremlinMA}. One remarkable point about Hausdorff's construction is that it does not require the Continuum Hypothesis or any other additional axioms, and that it works for the first uncountable cardinal $\aleph_1$. Compare it with the fact that, under the axiom $MA_{\aleph_1}$, two orthogonal sets of cardinalities $\aleph_0$ and $\aleph_1$ are always separated. Hausdorff's construction cannot be generalized to higher dimensions,

\begin{thm}\label{aleph1}
Under $MA_{\aleph_1}$, if $\{\Gamma_i : i<n\}$, $n>2$, are pairwise orthogonal families of subsets of $N$ of cardinality $\aleph_1$, then they are separated.
\end{thm}

For readers familiar with Martin's axiom: The idea is that one tries to force a separation for $\{\Gamma_0,\Gamma_1,\Gamma_2\}$ by considering approximations $(p_0,p_1,p_2)\in\Gamma_0\times\Gamma_1\times\Gamma_2$ such that $p_0\cap p_1\cap p_2 = \emptyset$. One sees that such approximations form a ccc (actually $\sigma$-2-linked) poset, noticing that once $p_0\cap p_1$, $p_0\cap p_2$, $p_1\cap p_2$ are frozen, the conditions become compatible. Further details, as well as generalized forms of Theorem~\ref{aleph1} can be found in \cite[Section 6]{multiplegaps}.\\

While Theorem~\ref{aleph1} rules out a part of the theory of gaps to be extended to higher dimensions, there is a class of gaps which happens to florish into a beautiful and intricated theory in the multidimensional setting: the analytic gaps. A family $I$ of subsets of the countable set $N$ can be viewed as a subset $I\subset 2^N = \{0,1\}^N$, and the set $2^N$ can be endowed with the product topology which makes it homeomorphic to the Cantor set. In this way, we can measure the complexity of a family $I$ by looking at its topological complexity as a subset of $2^N$. Let us consider some examples before going further. The family
$$I = \left\{x\subset \omega : \sum_{n\in x}n^{-1}\leq 2\right\}$$
is a closed family of subsets of $\omega$. Namely, if $x\not\in I$, then there is a finite set $F\subset\omega$ such that $\sum_{n\in F}n^{-1}>2$, and then $\{y : y\supset F\}$ is a neighborhood of $x$ disjoint from $I$. On the other hand, the family
$$J = \left\{x\subset \omega : \sum_{n\in x}n^{-1}< +\infty\right\}$$
is not closed; it is a little bit more complicated: $J$ is an $F_\sigma$-family, the countable union of closed families
$$J = \bigcup_{m<\omega} \left\{x\subset \omega : \sum_{n\in x}n^{-1}\leq m\right\}.$$
We think of closed families as the simplest ones, and a family $I$ is considered more complex if a more complicated expression is needed to obtain $I$ from closed families. In this sense, $F_\sigma$ families have a low complexity. We say that $I\subset 2^N$ is analytic if there exist\footnote{$\omega^{<\omega}$ is the set of finite sequences of natural numbers. Given $x=(x_0,x_1,\ldots)\in\omega^{<\omega}$ and $n<\omega$, $x|n = (x_0,\ldots,x_{n-1})$.} $\{I_s : s\in\omega^{<\omega}\}$ closed families such that
$$ (\star)\  I = \bigcup_{x\in\omega^\omega} \bigcap_{n<\omega}I_{x|n}$$ 
The class of analytic sets includes closed sets and open sets, and is closed under countable unions, countable intersections and \emph{Souslin operations}  (meaning that if we have $\{I_s : s\in\omega^{<\omega}\}$ analytic sets, and $I$ is as in $(\star)$ above, then $I$ is analytic). As an intuitive idea, a family $I$ is analytic if it has a \emph{reasonably simple} definition that allows it to be expressed in terms of closed families as in $(\star)$. The reader who is not familiar with the notion of analytic set is encouraged to visit \cite{Kechris}.\\

An $n_\ast$-gap $\Gamma = \{\Gamma_i : i<n\}$ is analytic if each $\Gamma_i$ is an analytic family. A theory of analytic gaps was initiated in~\cite{Todorcevicgap}. The gap constructed by Hausdorff was not analytic (he used transfinite induction on $\omega_1$, and this technique gets us out of the \emph{reasonable definitions} of the analytic world). Todorcevic found that analytic gaps had extra structure, and in particular some of the properties of Hausdorff's gaps can never hold for analytic gaps. Consider the following variation on the notion of separation:

\begin{dfn}The families $\{\Gamma_i : i<n\}$ are said to be countably separated if there exists a contable set ${C}$ such that whenever we pick $x_i\in\Gamma_i$ for $i\in n$, there exist $c_0,\ldots,c_{n-1}\in {C}$ such that $x_i\subset c_i$ and $\bigcap c_i = \emptyset$. An $n_\ast$-gap $\Gamma = \{\Gamma_i : i<n\}$ is called strong if it is not countably separated.
\end{dfn}

Todorcevic formulated his results as two dichotomies, that indicate that there are two critical examples of analytic 2-gaps: one which is strong and another which is not. Every analytic 2-gap \emph{contains} the critical 2-gap which is not strong, and every analytic strong 2-gap \emph{contains} the critical strong 2-gap. These results are essentially the two-dimensional case of Theorem~\ref{criticalgap} and Theorem~\ref{criticalstronggap} that we discuss next.

\section{Critical analytic gaps}\label{sectioncritical}

Given $n^{<\omega}$, the $n$-adic tree is the set of all finite sequences of numbers from $\{0,\ldots,n-1\}$. 
$$n^{<\omega} = \{ (s_0,\ldots,s_k) : s_0,\ldots,s_k\in \{0,\ldots,n-1\}\}$$
The empty sequence is considered as an element of the $n$-adic tree.\\

 The concatenation of two elements $s=(s_0,\ldots,s_p), t=(t_0,\ldots,t_q)\in n^{<\omega}$ is defined as $s^\frown t = (s_0,\ldots,s_p,t_0,\ldots,t_q)$. Sometimes, in abuse of notation, we write $s^\frown i$ for $s^\frown (i)$ when $i\in n$.\\

We say that $s\leq t$ if there exists $r$ such that $t = s^\frown r$. This is a partial order relation on the $n$-adic tree. A chain in the $n$-adic tree refers to a chain in this order, a set $X\subset n^{<\omega}$ such that for each $t,s\in X$, either $t\leq s$ or $s\leq t$.

\begin{dfn} Let $0\leq i< n$ be natural numbers, and $X\subset n^{<\omega}$ be a subset of the $n$-adic tree.

\begin{enumerate}
\item We say that $X$ is an $i$-chain if we can write $X = \{x_0,x_1,\ldots\}$ in such a way that $x_k ^\frown i \leq x_{k+1}$ for every $k$.
\item We say that $X$ is an $[i]$-chain if we can write $X = \{x_0,x_1,\ldots\}$ in such a way that $x_{k+1} = x_k ^\frown i ^\frown w_k$ where $\max(w_k)\leq i$.\\
\end{enumerate}  
\end{dfn}

\textbf{The critical $n$-gap} is the $n$-gap $\mathcal{C}^{n} = \{\mathcal{C}^n_i : i<n\}$, where $\mathcal{C}^n_i$ is the collection of all $[i]$-chains of the $n$-adic tree.\\

\textbf{The critical strong $n$-gap} is the $n$-gap $\mathcal{S}^{n} = \{\mathcal{S}^n_i : i<n\}$, where $\mathcal{S}^n_i$ is the collection of all $i$-chains of the $n$-adic tree.\\

A proof that $\mathcal{C}^n$ is an $n$-gap can be found in \cite[Lemma 1.15]{analyticmultigaps} and a proof that $\mathcal{S}^n$ is a strong $n$-gap can be found in \cite[Theorem 6]{multiplegaps}. In any case, these proofs are not complicated, we give some hints from which the reader can derive these facts as an excercise:

\begin{itemize}
\item In order to see that $\mathcal{C}^n$ is not separated, check that if $x\subset^\ast a$ for all $x\in\mathcal{C}^n_i$, then for every $t\in n^{<\omega}$ there exists $s>t$ such that $s^\frown r\in a$ for all $r\in\{0,1,\ldots,i\}^{<\omega}$.
\item In order to see that $\mathcal{S}^n$ is not countably separated, suppose that $C$ is a countable family of sets that countably separates $\mathcal{S}$. Then, for evey $x\in n^\omega$ there exist $c_0,\ldots,c_{n-1}\in C$ and $m<\omega$ such that\footnote{If $x=(x_0,x_1,\ldots)\in n^\omega$ and $s\in n^{<\omega}$, $s<x$ means that there exists $k$ such that $s=(x_0,\ldots,x_{k-1})$.} $\bigcap_{i<n}c_i=\emptyset$ and $\{t\in n^{<\omega} : |t|>m, t^\frown i <x\}\subset c_i$ for all $i$. By a Baire category argument, there exists $r\in n^{<\omega}$ so that one can take the same $m$ and the same $c_i$'s for all $x>r$, and this leads to a contradiction. 
\end{itemize}

These are gaps of very low complexity, actually each family $\mathcal{C}^n_i$ and $\mathcal{S}^n_i$ is closed. We call these gaps critical because any other analytic (strong) $n$-gap must contain the critical example inside.

\begin{thm}\label{criticalstronggap}
If $\Gamma$ is a strong analytic $n_\ast$-gap, then there exists a one-to-one function $\phi:n^{<\omega}\To N$ such that\footnote{We use the notation $\phi(\mathcal{I}) = \{\phi(X) : X\in\mathcal{\mathcal{I}}\}$, where $\phi(X) = \{\phi(s) : s\in 2^{<\omega}\}$.} $\phi(\mathcal{S}^n_i)\subset \Gamma_i$ for all $i<n$.
\end{thm}

\begin{thm}\label{criticalgap}
If $\Gamma$ is an analytic $n_\ast$-gap, then there exists a one-to-one function $\phi:n^{<\omega}\To N$ and a permutation $\varepsilon:n\To n$ such that $\phi(\mathcal{C}^n_i)\subset \Gamma_{\varepsilon(i)}$ for all $i<n$.
\end{thm}

Theorem~\ref{criticalgap} can be found in \cite[Theorem 1.14]{analyticmultigaps}, while Theorem~\ref{criticalstronggap} corresponds to \cite[Theorem 7]{multiplegaps}. The statement of \cite[Theorem 7]{multiplegaps} refers only to $n$-gaps, but the same proof works for $n_\ast$-gaps. The critical strong $n$-gap $\mathcal{S}^n$ is symmetric, in the sense that for every permutation $\varepsilon:n\To n$, the induced bijection $n^{<\omega}\To n^{<\omega}$ identifies $\mathcal{S}^n$ with $\{\mathcal{S}^n_{\varepsilon(i)} : i<n\}$. The critical $n$-gap $\mathcal{C}^n$ is, however, highly non-symmetric. For example, notice that $\mathcal{C}^n_0$ is just countably generated as every $[0]$-chain is contained in $\{s^\frown (0,\ldots,0)\}$ for some $s\in n^{<\omega}$ but $\mathcal{C}^n_i$ is not countably generated for $i>0$. Moreover, consider the following notions:

\begin{itemize}
\item We say that a family of sets $I$ is countably generated in another family $J$ if there exists a countable subfamily $D\subset J$ such that for every $x\in I$ there exists $y\in D$ such that $x\subset y$. 
\item If $I$ is a family of subsets of $N$, its orthogonal family is $I^\perp = \{y\subset N : \forall x \in I \ x\cap y=^\ast \emptyset\}$. 
\end{itemize}

The gap $\mathcal{C}^n$ is not invariant under any permutation because $\mathcal{C}^i_n$ is countably generated in $(\mathcal{C}^j_n)^\perp$ if and only if $i<j$. Indeed, if $i<j$, we have the countable family $$ D = \{ \{t^\frown r : r\in\{0,\ldots,j-1\}^{<\omega}\} : t \in n^{<\omega}\}$$
and on the other hand, it follows from Proposition~\ref{separationofcountable} that if two orthogonal families are not separated, then they cannot be countably generated in the orthogonal of each other.

When $n=2$, Theorem~\ref{criticalgap} can be stated more generally, as one side of the gap need not be analytic. This is \cite[Theorem 3]{Todorcevicgap}:

\begin{thm}\label{firstdichotomy}
Let $\Gamma = \{\Gamma_0,\Gamma_1\}$ is a gap such that $\Gamma_1$ is analytic and is not countably generated in $\Gamma_0^\perp$. Then there exists a one-to-one function $u:2^{<\omega}\To N$ such that $u(\mathcal{C}^n_i)\subset \Gamma_{i}$ for all $i=0,1$.
\end{thm}

The statement of \cite[Theorem 3]{Todorcevicgap} is a little different from Theorem~\ref{firstdichotomy}. To see that they are equivalent, it is enough to notice that the bijection $f:2^{<\omega}\To \omega^{<\omega}$ given by $g(\emptyset) = \emptyset$, $g(t^\frown 1) = g(t)^\frown 0$ and $g((t_1,\ldots,t_k)^\frown 0) = (t_1,\ldots,t_k+1)$ transforms the gap $\mathcal{C}^2$ into the gap $\tilde{\mathcal{C}}^2$, where $\tilde{\mathcal{C}}^2_0$ consists of the sets which are contained in the set of immediate successors of some $s\in\omega^{<\omega}$, and $\tilde{\mathcal{C}}^2_1$ consists of all chains of $\omega^{<\omega}$.\\

The critical strong 2-gap $\mathcal{S}^2$ is related to the phenomenon of so-called Luzin gaps. Suppose that you have families $\{a_x : x\in X\}$ and $\{b_x : x\in X\}$ of subsets of $\omega$ indicated in some uncountable set $X$ such that

\begin{enumerate}
\item $a_x\cap b_x = \emptyset$ for all $x$,
\item $0<|(a_x\cap b_y)\cup (a_y\cap b_x)| < \omega$ for all $x\neq y$.
\end{enumerate}

Then one can check that $\{a_x : x\in X\}$ and $\{b_x : x\in X\}$ form a strong 2-gap, cf. \cite[p. 56]{Todorcevicgap}. Gaps of this form are called Luzin gaps. The gap $\mathcal{S}^2$ is like that, as we can take $X = 2^\omega$, $a_x = \{s\in x : s^\frown 0< x\}$, $b_x = \{s\in x : s^\frown 1< x\}$. It is indeed a \emph{perfect Luzin gap} since it is parametrized by the Cantor set in a continuous way. \cite[Theorem 2]{Todorcevicgap} asserts that every analytic strong 2-gap contains a perfect Luzin gap, and therefore corresponds to the case $n=2$ of Theorem~\ref{criticalstronggap} above. A non-definable version of this result holds under the Open Coloring Axiom OCA (cf.~\cite{MoorePFA} for information on this axiom).

\begin{thm}[OCA]\label{OCA}
If $\{\Gamma_0,\Gamma_1\}$ is a strong 2-gap, then there exists a Luzin gap $\{a_x, b_x : x\in X\}$ such that $a_x\in \Gamma_0$, $b_x\in \Gamma_1$ for all $x\in X$.
\end{thm}

The proof of this is a direct application of OCA to the set of pairs $\{(a,b),(a',b')\}\subset \Gamma_0\times\Gamma_1$ such that $(a\cap b')\cup (a'\cap b)\neq\emptyset$. It is unclear if the higher-dimensional instances of Theorem~\ref{criticalstronggap} admit any analysis in terms of a Luzin-like condition that would allow a non-definable version as Theorem~\ref{OCA} above. One essential obstruction is that we cannot expect $\aleph_1$-generated $n$-gaps to play any role for $n>2$, by Theorem~\ref{aleph1}.\\

Theorems~\ref{criticalstronggap} and~\ref{criticalgap} state that, in a certain sense, every analytic (strong) $n_\ast$-gap contains a (strong) critical $n_\ast$-gap. However, the notion of \emph{being contained} that is implicit in these statements ignores important information. A finer notion is obtained by demanding the preservation of the orthogonals:

\begin{dfn}\label{gaporder}
Let $\Gamma$ and $\Delta$ be $n_\ast$-gaps on the sets $N$ and $M$ respectively. We say that $\Gamma\leq \Delta$ if\footnote{There are some variations of the order $\leq$ that lead essentially to the same theory, see \cite[Section 1.4]{analyticmultigaps}} there exists a one-to-one function $\phi:N\To M$ such that for every $i\in n$:
\begin{enumerate}
\item If $x\in \Gamma_i$ then $\phi(x)\in \Delta_i$
\item If $x\in\Gamma_i^\perp$, then $\phi(x)\in \Delta_i^\perp$
\end{enumerate}
\end{dfn}

When considering this finer relation between gaps, there is not anymore just one critical analytic $n_\ast$-gap below all others, but still there is a finite number of minimal $n_\ast$-gaps so that any analytic $n_\ast$-gap contains one of them.

\begin{thm}\label{minimalstronggaps}
Given $n<\omega$, there exists a finite list of closed strong $n_\ast$-gaps $\Gamma^1,\ldots,\Gamma^{q_n}$ such that for every analytic strong $n_\ast$-gap $\Gamma$ there exists $j$ such that $\Gamma^j \leq \Gamma$.
\end{thm}

\begin{thm}\label{minimalgaps}
Given $n<\omega$, there exists a finite list of closed $n_\ast$-gaps $\Gamma^1,\ldots,\Gamma^{p_n}$ such that for every analytic $n_\ast$-gap $\Gamma$ there exists $j$ such that $\Gamma^j \leq \Gamma$.
\end{thm}

Theorem~\ref{minimalgaps} is proven in \cite{IHES}. A less general version of Theorem~\ref{minimalstronggaps} is found in \cite{stronggaps}, but essentially the same proof works for the statement here. Of course, one can suppose that in the lists of gaps provided in these theorems we have that $\Gamma^i\not\leq \Gamma^j$ when $i\neq j$, and in this case we have the right to call the gaps in such a list \emph{the minimal analytic (strong) $n_\ast$-gaps}. Formally, we can introduce the following definitions:

\begin{dfn}
An analytic (strong) $n_\ast$-gap $\Gamma$ is said to be a minimal analytic (strong) $n_\ast$-gap if for every other analytic (strong) $n_\ast$-gap $\Delta$, if $\Delta\leq \Gamma$ then $\Gamma\leq\Delta$.
\end{dfn}

\begin{dfn}
Two minimal analytic (strong) $n_\ast$-gaps $\Gamma$ and $\Delta$ are equivalent if $\Gamma\leq\Delta$ and $\Delta\leq\Gamma$.
\end{dfn}

In this language, Theorems~\ref{minimalgaps} and \ref{minimalstronggaps} can be reformulated as follows:

\begin{thm} For every $n<\omega$, there exists only finitely many equivalence classes of minimal analytic (strong) $n_\ast$-gaps. Moreover, for every analytic (strong) $n_\ast$-gap $\Delta$ there exists a minimal analytic (strong) $n_\ast$-gap $\Gamma$ such that $\Gamma\leq \Delta$.
\end{thm}

The combination of Theorem~\ref{criticalgap} or Theorem~\ref{criticalstronggap} with Ramsey theoretic techniques provides a list like stated in Theorem~\ref{minimalgaps} or Theorem~\ref{minimalstronggaps}, but such a list is very redundant and a further combinatorial analysis to find the really minimal elements. Sections~\ref{sectionstrong} and~\ref{sectionweak} below explain the situation in the strong and general case respectively.

\section{First-move combinatorics of the $n$-adic tree and strong analytic gaps}\label{sectionstrong}

The level of an element of the $n$-adic tree $n^{<\omega}$ is defined as $|(s_0,\ldots,s_{p-1})| = p$. We introduce a well order $\prec$ on $n^{<\omega}$ given by $s\prec t$ if either $|s|<|t|$, or $|s|=|t| = p$ and $n^{p} s_0 + n^{p-1}s_1 + \ldots < n^p t_0 + n^{p-1}t_1 + \ldots$.\\

The meet of $s,t\in n^{<\omega}$ is the infimum of $s$ and $t$ in the order $\leq$. That is, $s\wedge t$ it is the largest element $r$ such that $r\leq t$ and $r\leq s$.\\

If $t\leq s$, then $s\setminus t$ is the element $r\in n^{<\omega}$ such that $t^\frown r = s$.\\

By the \emph{first-move combinatorics} of the $n$-adic tree, we refer to the combinatorial problems about subsets of the $n$-adic tree where the only relevant structure is given by the the meet function $t\wedge s$, the order $\prec$ and the `first move' between comparable nodes (if $t<s$, knowing which is the $i\in n$ such that $t^\frown i\leq s$).\\
     
We need a few definitions to make this idea precise. A set $A\subset n^{<\omega}$ will be said to be meet-closed if $t\wedge s\in A$ whenever $t\in A$ and $s\in A$. The meet-closure of $A$ is the intersection of all  meet-closed sets that contain $A$ and is denoted by $\langle\langle A\rangle\rangle$.\\

A bijection $f:A\To B$ is a first-move-equivalence  if it is the restriction of a bijection $f:\langle\langle A\rangle\rangle\To \langle\langle B\rangle\rangle$ such that for every $t,s\in \langle\langle A\rangle\rangle$
\begin{enumerate}
\item $f(t\wedge s) = f(t)\wedge f(s)$
\item $f(t) \prec f(s)$ if and only if $t\prec s$
\item If $i\in n$ is such that $t^\frown i \leq s$, then $f(t)^\frown i \leq f(s)$.
\end{enumerate}

The sets $A$ and $B$ are called first-move-equivalent if there is a first-move-equivalence between them. In this case, we write $A\approx B$. In this language, the following is a consequence of Milliken's partition theorem for trees~\cite{Milliken}:

\begin{thm}\label{strongRamsey}
Fix a set $A_0\subset n^{<\omega}$. If we color the set $\{A\subset n^{<\omega} : A\approx A_0\}$ into finitely many colors in a Baire-measurable way\footnote{To color the set $\mathcal{A} = \{A\subset n^{<\omega} : A\approx A_0\}$ into finitely many colors, just means to give a function $\mathcal{A}\To F$ where $F$ is finite set whose elements we call \emph{colors}. The Baire-measurabily refers to the fact that we can view $\mathcal{A}$ as a subset of the Cantor set $2^{n^{<\omega}}\equiv 2^\omega$, with its Baire $\sigma$-algebra generated by the open sets and the meager sets \cite[Section 8F]{Kechris}.}, then there exists a set $T\subset n^{<\omega}$ such that $T\approx n^{<\omega}$ and all the sets from $\{A\subset T : A\approx A_0\}$ have the same color.
\end{thm}

The above statement can be also proven as a corollary of \cite[Theorem 1.5]{analyticmultigaps}, we will explain the argument later in Section~\ref{sectionweak} after Theorem~\ref{weakRamsey}. Finding large monochromatic sets for colorings on a given structure is the general theme of Ramsey theory. We refer to \cite{Ramsey} for an exposition on the Ramsey theory of countable infinite structures, including the Milliken's theorem mentioned above and many related results.\\

The notion of $i$-chain that appeared in the definition of the critical strong $n$-gap can be rephrased saying that set $X\subset n^{<\omega}$ is an $i$-chain if and only if $$X\approx \{ (i), (ii), (iii), (iiii), \ldots\}$$
We can generalize this notion to include some kinds of sets which are not chains. If $i,j\in n$, a set $X$ is called an $(i,j)$-comb if 
$$X\approx \{ (j), (iij), (iiiij), (iiiiiij), \ldots\}$$
In this way, an $i$-chain is the same as an $(i,i)$-comb. When $i\neq j$, $(i,j)$-combs are antichains (that is, no two different elements are comparable in the order $\leq$). The $(i,j)$-combs are the minimal first-move-equivalence classes of infinite subsets of $n^{<\omega}$, in the sense that the following two facts hold:
\begin{enumerate}
\item If $X$ is an $(i,j)$-comb, then every infinite subset of $X$ is again an $(i,j)$-comb.
\item Every infinite subset $X\subset n^{<\omega}$ contains a further subset $Y\subset X$ which is an $(i,j)$-comb for some $i,j\in n$ \cite[Lemma 7]{stronggaps}
\end{enumerate}

For $S\subset n\times n$, let $\Gamma_S$ be the family of all subsets of $n^{<\omega}$ which are $(i,j)$-combs for some $(i,j)\in S$. The following two results state the connection between analytic strong $n_\ast$-gaps and combs of the $n$-adic tree: every analytic strong $n_\ast$-gap contains - in the sense of the order $\leq$ of Definition~\ref{gaporder}- a strong gap made of combs.

\begin{thm}
If $\{S_i : i<n\}$ are nonempty subsets of $m\times m$, then $\{\Gamma_{S_i} : i\in n\}$ are closed families of subsets of $m^{<\omega}$ which are not countably separated.
\end{thm}

The intersection of two sets of different types is always finite (indeed, either empty or a singleton). Therefore $\{\Gamma_{S_i} : i\in n\}$ is an $n_\ast$-gap when $\bigcap_{i<n}S_i=\emptyset$, and it is an $n$-gap when the sets $S_i$ are pairwise disjoint. 

\begin{thm}\label{standardbelowstrong}
If $\Delta$ is a strong $n_\ast$-gap, then there exist $\{S_i : i<n\}$ disjoint subsets of $n\times n$, such that $\{\Gamma_{S_i} : i\in n\} \leq \Delta$. Moreover, the sets $S_i$ can be chosen so that $(i,i)\in S_i$.
\end{thm}

Theorem~\ref{minimalstronggaps} follows from these results, as there are only finitely many choices for sets $\{S_i : i<n\}$. Theorem~\ref{standardbelowstrong} is proven combining Theorem~\ref{criticalstronggap} and Theorem~\ref{strongRamsey}: We start with an injective function $u:n^{<\omega}\To N$ given by Theorem~\ref{criticalstronggap} and then for every $i,j,k$ we color the $(i,j)$-combs $X$ into two colors depending whether $u(X)\in \Delta_k$ or $u(X)\not\in\Delta_k$, and we successively apply Theorem~\ref{strongRamsey}, so that at the end, we get a restriction of $u$ to a set $A\approx n^{<\omega}$ that witnesses that $\{\Gamma_{S_i} : i\in n\}|_A \leq \Delta$.\\

We arrive to the conclusion that, if we are interested in studying properties of strong analytic $n_\ast$-gaps that can be reduced using the order $\leq$, then we must study gaps of the form $\{\Gamma_{S_i} : i<n\}$ made of combs as above. We should understand in particular when we have $\{\Gamma_{S_i}:i<n\}\leq \{\Gamma_{S'_i}:i<n\}$ for different families $S_i$ and $S'_i$ of combs. This requires understanding what kind of transformations $\varepsilon:n^2\To m^2$ are induced by one-to-one functions $\phi:n^{<\omega}\To m^{<\omega}$, in the sense that $\phi$ is a one-to-one map that sends $(i,j)$-combs to $\varepsilon(i,j)$-combs.

\begin{thm}
For every function $\phi:n^{<\omega}\To m^{<\omega}$ there exists $T\subset n^{<\omega}$ such that $T\approx n^{<\omega}$ and $\phi(A)\approx \phi(B)$ whenever $A\cup B\subset T$ and $A\approx B$.
\end{thm}

As a corollary, if $\varepsilon:n^2\To m^2$ is induced by some injective function $\phi:n^{<\omega}\To m^{<\omega}$, then we can suppose that $\phi$ satisfies that $\phi(A)\approx \phi(B)$ whenever $A\approx B$. All the relevant information about a function $\phi$ satisfying this regularity property is determined by the $\approx$-equivalence class of the family $$\{\phi(\emptyset),\phi(0),\phi(1),\ldots,\phi(n-1)\}.$$
For technical reasons, the behavior of $\phi$ is better studied by looking at a \emph{normalization} of this family to a single level, like
$$\{\phi(\emptyset),\phi(0)|_m,\phi(1)|_m,\ldots,\phi(n-1)|_m\}$$ 
where $m=|\phi(n-1)|$. This family is renamed in \cite{stronggaps} as $$\{e(\infty),e(0),e(1)\ldots,e(n-1)\}.$$ A non-degeneration argument allows to suppose that $e(i)\neq e(j)$ if $i\neq j$. From this family, we can recover all the information needed about $\phi$ because we have that $\varepsilon(i,j) = (u,v)$ if
\begin{itemize}
\item either $i\neq j$, $t = e(i)\wedge e(j)$, $e(i) \geq t^\frown u$ and $e(j) \geq t^\frown v$,
\item or $i=j$, $t = e(\infty)\wedge e(i)$, $e(i)\geq t^\frown u$, $e(\infty)\geq t^\frown v$.
\end{itemize}

Conversely, each one-to-one function $e:\{\infty,0,1,\ldots,n-1\}\To m^{<\omega}$ where $|e(\infty)|<|e(0)|=|e(1)|=\cdots=|e(n-1)|$ is associated to a function $\phi$.\\

Let us illustrate the way of reasoning by looking at 2-gaps. There are four kind of combs in the dyadic tree: $(0,0)$, $(0,1)$, $(1,0)$ and $(1,1)$. We know that every analytic strong 2-gap contains one of the form $\{\Gamma_{S_0},\Gamma_{S_1}\}$ where $(i,i)\in S_i$. There are $3^2 = 9$ ways of distributing the other two combs $(0,1)$ and $(1,0)$ between $\Gamma_{S_0}$ and $\Gamma_{S_1}$. This provides a list nine 2-gaps such that every analytic strong 2-gaps contains one of them. However, some of these 9 gaps are comparable to others. There are six equivalence classes of minimal analytic strong 2-gaps, whose representatives are given in the following table of comb distributions:
 
\begin{center}
$
\begin{array}{|c|l|l|}
\hline
 & \Gamma_0  & \Gamma_1\\
\hline
1 & (0,0) , (0,1)         &  (1,1) , (1,0) \\
\hline
2 & (0,0)  & (1,1) \\
\hline
3 & (0,0)  & (1,1) , (0,1), (1,0)\\
\hline
3^\ast & (0,0) , (0,1) , (1,0) & (1,1)\\
\hline
4 & (0,0) & (1,1) , (1,0) \\
\hline
4^\ast & (0,0) , (0,1) & (1,1)\\
\hline
\end{array}
$ 
\end{center}

The gaps enumerated as $k$ and $k^\ast$ is because one is the permutation of the other. In order to check that this list is correct we need to check two things: first, that each of the three gaps that have been excluded from the list contain some gap from the list, and second that the gaps in the list are incomparable to each other. We check a particular case of each task as an illustration. The possibility $\tilde{S}_0 = \{(0,0),(1,0)\}$, $\tilde{S}_1 = \{(1,1)\}$ does not appear in the list: we check that the gap number $4^\ast$ is already below that gap. We need to find an injective map $\phi:2^{<\omega}\To 2^{<\omega}$ that takes combs from $S_0(4^\ast)$ to combs from $\tilde{S}_0$, combs from $S_1(4^\ast)$ to combs from $\tilde{S}_1$, and combs out of $S_0 \cup S_1(4^\ast)$ to combs out of $\tilde{S}_0\cup \tilde{S}_1$ (the last condition is necessary for the preservation of the orthogonals). We can suppose that $\phi$ satisfies $A\approx B \Rightarrow \phi(A)\approx \phi(B)$ and such a function is determined by its associated family $\{e(\infty),e(0),e(1)\}$. Consider the function $\phi$ for which $e(\infty) = (0)$, $e(0) = (11)$, $e(1) = (01)$. The explicit formula for this $\phi$ is
$$\phi(i_0,i_1,\ldots,i_k) = (0,1-i_0,0,1-i_1,\ldots,0,1-i_k)$$
and it is easy to check that $\phi(0,0)=(1,0)$, $\phi(1,1)=(1,1)$, $\phi(0,1) = (0,1)$, $\phi(1,0)=(0,1)$, and hence $\phi$ is as desired. Now, let us check that the gap number 3 is not below gap number 4. Suppose that there exists an injective function $\phi:2^{<\omega}\To 2^{<\omega}$ that witnesses that relation. We look at its associated $\{e(\infty),e(0),e(1)\}$. So let $t = e(0)\wedge e(1)$, $t^\frown i = e(0)$, $t^\frown j = e(1)$, $i\neq j$. Then $\phi$ takes $(0,1)$-combs to $(i,j)$-combs and $(1,0)$-combs to $(j,i)$-combs. It is impossible then that $\phi$ witness that gap 3 is below gap 4.\\

The technique of analyzing each possible injective function $\phi:n^{<\omega}\To m^{<\omega}$ by means of its associated set $\{e(\infty),e(0),e(1)\ldots,e(n-1)\}$ is quite effective and it allows to compute the list of minimal analytic strong $n$-gaps for every $n$. Each equivalence class is determined by certain parameters $(A,B,C,D,E,\psi,\mathcal{P},\gamma)$. In the case $n=3$, there are $4^6=4096$ ways of distributing the combs of $3^{<\omega}$ into three pairwise disjoint sets where $(0,0)\in S_0$, $(1,1)\in S_1$, $(2,2)\in S_2$. But there are only 31 (9 up to permutation) equivalence classes of minimal analytic strong 3-gaps. We refer to \cite{stronggaps} for further details.

\section{Record combinatorics of the $n$-adic tree and analytic gaps}\label{sectionweak}

The content of this section has some analogy to that of Section~\ref{sectionstrong}, but with additional difficulties, so we recommend the reader to understand Section~\ref{sectionstrong} first. We shall make use of all the notations about the $n$-adic tree introduced at the beginning of Sections~\ref{sectioncritical} and~\ref{sectionstrong}.

In the record combinatorics, we are interested in problems where the relevant structure is given by the meet operation $t\wedge s$ of two nodes $t,s$, the order $t\prec s$, and the record history from a node $t$ to a larger node $s>t$. By the record history we mean the sequence of nodes $$record(t,s) = \{t=t_0<t_1<\cdots<t_{k+1}=s\}$$ and the sequence of integers $$0\leq m_0<m_1<\cdots<m_k<n$$ such that $t_i^\frown m_i \leq s$ and $m_i = \max(t_{i+1}\setminus t_i)$ for all $i=0,\ldots,k$. That is, we start climbing up from $t$ to $s$ and we take note of each time that we reach an integer which is strictly larger than all the integers we saw before (each time that we \emph{break a record}). We are considering thus a finer structure than in the first-move case where only the first record $m_0$ was relevant to the structure.\\

More precise definitions: A set $A\subset n^{<\omega}$ is said to be record-closed if the following two properties hold:
\begin{enumerate}
\item $t\wedge s$ whenever $t\in A$ and $s\in A$,
\item $record(t,s)\subset A$ whenever $t\in A$, $s\in A$, $t<s$.
\end{enumerate}
The record-closure of $A$, denoted by $\langle A\rangle$ is the intersection of all record-closed sets that contain $A$. A bijection $f:A\To B$ is a record-equivalence  if it is the restriction of a bijection $f:\langle A\rangle\To \langle B\rangle$ such that for every $t,s\in \langle A\rangle$
\begin{enumerate}
\item $f(t\wedge s) = f(t)\wedge f(s)$
\item $f(t) \prec f(s)$ if and only if $t\prec s$
\item If $i\in n$ is such that $t^\frown i \leq s$, then $f(t)^\frown i \leq f(s)$.
\end{enumerate}

The sets $A$ and $B$ are called record-equivalent if there is a record-equivalence between them. In this case, we write $A\sim B$. A partition theorem analogous to that of Milliken holds in this context:

\begin{thm}\label{weakRamsey}
Fix a set $A_0\subset n^{<\omega}$. If we color the set $\{A\subset n^{<\omega} : A\sim A_0\}$ into finitely many colors in a Baire-measurable way, then there exists a set $T\subset n^{<\omega}$ such that $T\sim n^{<\omega}$ and all the sets from $\{A\subset T : A\sim A_0\}$ have the same color.
\end{thm}

A proof of this theorem is found in \cite[Section 1]{analyticmultigaps}. We promised a proof of Theorem~\ref{strongRamsey} from Theorem~\ref{weakRamsey}. This is actually very simple: It is enough to consider the function $\psi:n^{<\omega}\To n^{<\omega}$ given by $\psi(s_0,s_1,\ldots,s_k) = (s_0,n-1,s_1,n-1,\ldots,s_k,n-1)$. This function has the property that if $A\approx B$, then $A\approx \psi(A)\sim \psi(B)\approx B$. It is now a straightforward excercise to deduce Theorem~\ref{strongRamsey} from Theorem~\ref{weakRamsey} using $\psi$.\\

Similarly as in the strong case, the $[i]$-chains that appeared in the critical $n$-gap can be reinterpreted now in the following way: A set $X\subset n^{<\omega}$ is an $[i]$-chain if and only if
$$ X\sim \{(i), (ii), (iii), (iiii),\ldots\}$$ 
The next step is to identify the anologous of the $(i,j)$-combs in the first-move combinatorics, that is, the minimal record-equivalence classes of infinite sets. This equivalence classes are not parametrized just by couples of integers $(i,j)$ but by more complicated beings that we call \emph{types}. A type $\tau = (\tau^0,\tau^1,\triangleleft)$ in $n^{<\omega}$ consists of two sets $\tau^0\subset n$ and $\tau^1\subset n$, together with a total order relation $\triangleleft$ on $\tau^0\times\{0\}\cup\tau^1\times\{1\}$ such that:
\begin{itemize}
\item $\tau^0\neq \emptyset$,
\item $\min(\tau^0) \neq \min(\tau^1)$,
\item $(k,i)\triangleleft (k',i)$ whenever $i\in\{0,1\}$ and $k<k'$,
\item $(\max(\tau^0),0)$ is the maximal element in the order $\triangleleft$.
\end{itemize}
Given a type $\tau$ where $\tau^0 = \{a_1^0<\ldots<a_p^0\}$ and $\tau^1 = \{a^1_1<\ldots<a^1_q\}$, consider
\begin{eqnarray*}
u^\tau &=& (a^0_1,\ldots,a^0_1, a^0_2,\ldots,a^0_2,\ \ldots\ , a_p^0,\ldots, a_p^0)\\
v^\tau &=& (a^1_1,\ldots,a^1_1, a^1_2,\ldots,a^1_2,\ \ldots\ , a_q^1,\ldots, a_q^1)
\end{eqnarray*}

where the number of times $r(k,i)$ that $a^i_k$ is repeated satisfies that $r(k,i) > 2^{r(k',j)}$ whenever $(a^i_k,i) \triangleright (a^j_{k'},j)$. A set $X$ is of type $\tau$ if
$$X \sim \{v^\tau, {u^\tau}^\frown v^\tau, {u^\tau}^\frown {u^\tau}^\frown v^\tau,\ldots\}$$
Again, these happen to be the minimal record-equivalence classes of infinite sets: on the one hand, every infinite subset of a set of type $\tau$ has also type $\tau$, and on the other hand every infinite set contains an infinite subset which is of type $\tau$ for some type $\tau$. We represent a type as two rows of integers between square-brackets, the lower sequence represents $\tau^0$, the upper sequence represents $\tau^1$ and the order $\triangleleft$ is read from left to right. Thus, $\tau = [^{23}{}_{1}{}^4{}_2]$ means that $\tau^0 = \{1,2\}$, $\tau^1=\{2,3,4\}$ and $(2,1)\triangleleft (3,1)\triangleleft (1,0)\triangleleft (4,1)\triangleleft (2,0)$. If there is only one row, it means that $\tau^1=\emptyset$ and the row represents $\tau^0$. When $\tau^1=\emptyset$, the sets of type $\tau$ are chains, otherwise they are antichains. The sets of type $[i]$ are precisely the $[i]$-chains appearing in the critical analytic $n$-gap. We denote by $\mathfrak{T}_n$ the set of all types in $n^{<\omega}$.\\

For $S$ a set of types in $n^{<\omega}$, let now $\Gamma_S$ be the family of all subsets of $n^{<\omega}$ which are of type $\tau$ for some $\tau\in S$. Notice that $\Gamma_S$ is a closed family, in particular analytic.

\begin{thm}
If $\{S_i : i<n\}$ are nonempty sets of types in $m^{<\omega}$, then the families $\{\Gamma_{S_i} : i\in n\}$ are not separated.
\end{thm}

If $X$ is of type $\tau$ and $Y$ is of type $\sigma$, $\tau\neq \sigma$, then $|X\cap Y|\leq 3$. Hence, if the $\bigcap_{i<n}S_i=\emptyset$, then $\{\Gamma_{S_i} : i<n\}$ is an analytic $n_\ast$-gap. If the ${S_i}$'s are pairwise disjoint, then $\{\Gamma_{S_i} : i<n\}$ is an analytic $n$-gap. Similarly as in the strong case, now the fact is that every analytic $n$-gap contains an $n$-gap made of types in the $n$-adic tree.

\begin{thm}\label{standardbelowgap}
If $\Delta$ is an $n_\ast$-gap, then there exist $\{S_i : i<n\}$ sets of types in $n^{<\omega}$, such that $\{\Gamma_{S_i} : i\in n\} \leq \Delta$. Moreover, the sets $S_i$ can be chosen so that $[i]\in S_i$.
\end{thm}

Analogous arguments to those that we exposed after Theorem~\ref{standardbelowstrong} show that Theorem~\ref{standardbelowgap} follows from Theorem~\ref{criticalgap} and Theorem~\ref{weakRamsey}.\\

There are eight types in the dyadic tree $2^{<\omega}$, namely $[0]$, $[1]$, $[01]$, $[^0{}_1]$, $[^{01}{}_1]$, $[^1{}_0]$, $[^1{}_{01}]$ and $[_0{}^1{}_1]$. Theorem~\ref{standardbelowgap} states that any analytic 2-gap contains a permutation of one which is obtained by putting sets of type $[0]$ on one side, sets of type $[1]$ on the other side, and distributing the rest of types in whatever way. That means a list of $2\cdot 3^6 = 1458$ gaps so that any analytic gap contains one of them. But there are actually only nine equivalence classes of minimal analytic 2-gaps:\\

\begin{center}
$
\begin{array}{|l|l|l|}
\hline
 & \Gamma_0  & \Gamma_1\\
\hline
1 & [0]         &  \text{all other types} \\
\hline
1^\ast &  \text{all other types} & [0] \\
\hline
2 & [0]  & [1]\\
\hline
2^\ast & [1]  & [0]\\
\hline
3 & [0]  & [1] , [01]\\
\hline
3^\ast & [1], [01]  & [0]\\
\hline
4 & [0] , [01] & [1] \\
\hline
5 & [0] & [1] , [01] , [^1 {}_0 {}_1]\\
\hline
5^\ast & [1] , [01] , [^1 {}_0 {}_1] & [0]\\
\hline
\end{array}
$ 
\end{center}

In order to refine the information provided by Theorem~\ref{standardbelowstrong} to obtain reduced lists of minimal gaps, or in general to obtain more precise facts about analytic gaps, it is necessary to understand when we have $\Gamma\leq\Delta$ for gaps given by types, and this means understanding what kind of functions $\bar{\phi}:\mathfrak{T}_n\To \mathfrak{T}_m$ are induced by injective functions $\phi:n^{<\omega}\To m^{<\omega}$ in the sense that $\phi$ sends sets of type $\tau$ to sets of type $\bar{\phi}\tau$.

\begin{thm}\label{normalembedding}
For every function $\phi:n^{<\omega}\To m^{<\omega}$ there exists $T\subset n^{<\omega}$ such that $T\sim n^{<\omega}$ and $\phi(A)\sim \phi(B)$ whenever $A\cup B\subset T$ and $A\sim B$.
\end{thm}

The functions satisfying the regularity condition above are called normal embeddings. When $\phi$ is a normal embedding, we denote by $\bar{\phi}$ its action on types, so that $\phi$ sends sets of type $\tau$ to sets of type $\bar{\phi}\tau$. The situation is not as simple as in the strong case, and in particular, the equivalence class of $\{\phi(\emptyset),\phi(0),\ldots,\phi(n-1)\}$ does not determine the whole behavior of a normal embedding $\phi$. A number of facts about normal embeddings are proven in \cite{analyticmultigaps}. Just to get a flavor, for a type $\tau$ consider $\max(\tau) = \max(\tau^0\cup \tau^1)$ the largest integer appearing in the type $\tau$, then one can prove that if $\max(\tau)\leq \max(\sigma)$, then $\max(\bar{\phi}\tau) \leq \max(\bar{\phi}\sigma)$, and this is a basic fact needed in many computations. A variety of peculiar classes (top-combs, top$^2$-combs, chains...) and characteristics ($\max(\tau)$, $strength(\tau)$...) of types show up in studying this kind of combinatorics.\\

As an example of how things work, let us briefly explain why some of the types in $2^{<\omega}$ never appeared in the list of minimal 2-gaps. A type $\tau$ is a top-comb type if the second number from the right in its representation exists and lives in the upper row. Equivalently, we can say that $(\max(\tau^1),1)$ is the second-largest element of the order $\triangleleft$. A type $\tau$ dominates a type $\sigma$ if $\tau$ is a top-comb and $\max(\sigma)\leq \max(\tau^1)$. Domination is characterized in terms of normal embeddings as follows:

\begin{thm}\label{dominationthm}
Let $\tau_0,\tau_1$ be types in $m^{<\omega}$. The following are equivalent:
\begin{enumerate}
\item $\tau_{1}$ dominates $\tau_0$,
\item There exists a normal embedding $\phi:2^{<\omega}\To m^{<\omega}$ such that $\bar{\phi}[0] = \tau_0$, and $\bar{\phi}\sigma = \tau_1$ for all types $\sigma \neq [0]$.
\end{enumerate}
\end{thm}

The types $[^{01}{}_1]$ and $[^1{}_0]$ dominate any other type in $2^{<\omega}$. So the above result shows that if one of these two types appears somewhere in the gap $\Gamma$, then either the gap number 1 or number $1^\ast$ from the list of minimals above will be below $\Gamma$. This means that the original list of $2\cdot 3^6 = 1458$ gaps can be refined to the $2\cdot 3^4 = 162$ gaps that exclude $[^{01}{}_1]$ and $[^1{}_0]$.

The idea behind Theorem~\ref{dominationthm} can be understood by analyzing the particular case when $\tau_0 = [0]$ and $\tau_1 = [_0{}^1{}_1]$. The normal embedding $\phi:2^{<\omega}\To 2^{<\omega}$ for these two types is constructed in the following way: Fix $\{x_0,x_1,x_2,\ldots\}$ a set of type $\tau_1 = [_0{}^1{}_1]$, and then define $\phi$ so that $$\phi(s^\frown 1^\frown 0^\frown \cdots^\frown 0) = x_{n_s}{}^\frown 0^\frown\cdots^\frown 0$$
(The number $n_s$ and the number of repetitions of 0 are defined inductively so that $\phi(t)\prec \phi(t')$ if $t\prec t'$). This satisfies that $\bar{\phi}[0]=[0]$ and $\bar{\phi}\sigma = [_0{}^1{}_1]$ for all other types. Notice that if we consider a similar construction for the non-top-comb type $[^1{}_{01}]$ instead, then we will obtain that $\bar{\phi}\sigma = [^1{}_{01}]$ for some types and $\bar{\phi}\sigma = [_0{}^1{}_1]$ for other types.

\section{Breaking gaps: jigsaws and clovers}\label{sectionbreak}

When we have a mutliple gap $\{\Gamma_i : i\in n\}$ we can consider its subgaps, that is, the gaps of the form $\{\Gamma_i : i\in A\}$ where $A\subset n$. We can also consider restrictions $\{\Gamma_i|_a : i\in A\}$ where $a\subset \omega$, $\Gamma_i|_a = \{x\in\Gamma_i: x\subset a\}$, that may be still gaps or may become separated. The behavior of these operations can be quite different for different kinds of gaps.

\begin{dfn}
Let $\Gamma = \{\Gamma_i : i<n\}$ be an $n$-gap, and let $B\subset n$. We say that $\Gamma$ is $B$-broken if there exists an infinite set $M\subset N$ such that $\{\Gamma_i|_M : i\in B\}$ is a gap, but $M\in\Gamma_i^\perp$ for $i\not\in B$.
\end{dfn}

Let us consider two examples to illustrate this concept.\\

\begin{enumerate}
\item First, consider the critical 3-gap $\mathcal{C}^3$ in $3^{<\omega}$. Can this 3-gap be $\{0,1\}$-broken? Yes, it is enough to take $M = 2^{<\omega}\subset 3^{<\omega}$. On the one hand, $M\in (\mathcal{C}^3_2)^\perp$ because a set of type $[2]$ intersects $2^{<\omega}$ in at most one point. On the other hand $\{\mathcal{C}^3_0|_{2^{<\omega}},\mathcal{C}^3_1|_{2^{<\omega}}\} = \mathcal{C}^2$ is still a gap.\\

\item Now, consider another example $\Delta = \{\Delta_i : i<3\}$ in $2^{<\omega}$, where $\Delta_0$ is the family of all sets of type $[0]$, $\Delta_1$ is the family of all sets of type $[1]$, and $\Delta_2$ is the family of all sets of type $[01]$. This time, the 3-gap $\Delta$ cannot be $\{0,1\}$-broken. Let us sketch a proof. Suppose that we have $M\subset 2^{<\omega}$ that \emph{breaks} $\Delta$. By Theorem~\ref{standardbelowgap}, we can find a gap $\Delta'\leq \{\Delta_0|_M, \Delta_1|_M\}$ of the form $\Delta' = \{\Gamma_{S_0},\Gamma_{S_1}\}$ where $S_0$ and $S_1$ ares sets of types with $[\varepsilon 0]\in S_0$, $[\varepsilon 1]\in S_1$ for some permutation $\varepsilon:2\To 2$. The fact that $\Delta'\leq \{\Delta_0|_M, \Delta_1|_M\}$ must be witnessed by a certain function $\phi:2^{<\omega}\To M\subset 2^{<\omega}$, which by Theorem~\ref{normalembedding} can be supposed to be a normal embedding. Then $\phi$ must send sets of type of $[\varepsilon 0]$ onto sets of type $[0]$, and sets of type $[\varepsilon 1]$ onto sets of type $[1]$. We mentioned the property that normal embeddings satisfy $$\max(\sigma)\leq \max(\tau) \Rightarrow \max(\bar{\phi}\sigma)\leq \max(\bar{\phi}\tau)$$ Hence $\varepsilon 0 = 0$ and $\varepsilon 1 = 1$. It is an easy exercise that if $\phi$ preserves sets of type $[0]$ and sets of type $[1]$, then it also preserves sets of type $[01]$, and this contradicts that $M=\phi(2^{<\omega})\in\Delta_3^\perp$.\\
\end{enumerate}

The argument for the first example shows something more general: the critical $n$-gap $\mathcal{C}^n$ can be $B$-broken for all $B\subset n$, just by taking $M=B^{<\omega}$. Since this happens to be a minimal $n$-gap, it is actually a \emph{jigsaw}, according to the following definition:

\begin{dfn}
We say that an $n$-gap $\Gamma$ is a jigsaw if for every $B\subset A\subset n$ and for every $M\subset N$, if $\{\Gamma_i|_M : i\in A\}$ is not separated, then it can be $B$-broken. 
\end{dfn}

Jigsaws are multiple gaps that can be broken \emph{everywhere and in every way}, and we have just shown that there are very natural analytic examples of them. The critical strong $n$-gap is also a jigsaw. Anoter example of analytic jigsaw which is moreover \textit{dense} is the minimal analytic $n$-gap $\{\Gamma_i : i<n\}$ in $n^{<\omega}$, where $\Gamma_i$ is the family of all sets of type $\tau$ with $\max(\tau)=i$. We highlight the definition of dense gap since it will be important in later discussions:

\begin{dfn}
An $n$-gap $\Gamma=\{\Gamma_i : i<n\}$ in $N$ is dense if every infinite subset of $N$ contains an infinite set from $\bigcup_{i\in n}\Gamma_i$.
\end{dfn}

When an $n$-gap is made of types in $m^{<\omega}$, density means that each type of $m^{<\omega}$ is included in some $\Gamma_i$.\\

Let us look now at the second example $\Delta$ that we proposed. In that case, the argument for showing that $\Delta$ cannot be $\{0,1\}$-broken does not generalize to other pairs. Indeed, the reader can try to check as an excercise that $\Delta$ can be $\{0,2\}$-broken and also $\{1,2\}$-broken. Is it possible to produce an $n$-gap which cannot be $B$-broken for any $B\subset n$, $|B|\geq 2$?. Such $n$-gaps are called \emph{clovers}, and yes, here is an example of a clover \cite[Proposition 19]{multiplegaps}: Consider $2^\omega = \bigcup_{i<n}X_i$ a decomposition of the Cantor set into Bernstein sets. Remember that a set $X\subset 2^\omega$ is called Bernstein if both $C\cap X$ and $C\setminus X$ are uncountable for every uncountable compact subset $C\subset 2^\omega$. Such sets can be constructed by transfinite induction, cf.\cite[Example 8.24]{Kechris}. Each point $x=(x_0,x_1,\ldots)\in 2^\omega$ can be identified with  the corresponding branch $b_x = \{(x_0,\ldots,x_{k-1}) : k<\omega\}\subset 2^{<\omega}$. The clover is then
$$\tilde{\Delta} = \{\tilde{\Delta}_i = \{b_x : x\in X_i\} : i<n\}.$$
This clover is not an analytic $n$-gap: the construction of Bernstein sets requires transfinite induction, which means a too complicated definition to be analytic. Can we construct an analytic clover? No, we cannot. All analytic $n$-gaps can be $B$-broken for \emph{many} subsets $B\subset n$. We do not have a precise meaning for the word \emph{many}, and it is actually an interesting problem to characterize those families $\mathcal{B}$ of subsets of $n$ for which there is an analytic $n$-gap which can be $B$-broken if and only if $B\in\mathcal{B}$. We need to improve our understanding of the \emph{record combinatorics} of the $n$-adic to solve such a question. In the case of 3-gaps we do have such a characterization, and says that the two examples that we presented at the beginning of this section are the only possibilities:

\begin{thm}\label{3break}
Let $\Gamma = \{\Gamma_i : i<3\}$ be an analytic 3-gap. Then $\Gamma$ can be $B$-broken for at least two out of the three sets $B\subset \{0,1,2\}$ of cardinality 2.
\end{thm}

The proof of a theorem like this runs as follows: we can suppose that $\Gamma$ is a minimal analytic $3$-gap, hence made of types. The 61 types of the 3-adic tree, we can be distributed into $4^{61}$ ways. Using the combinatorial techniques explained in Section~\ref{sectionweak} for the study of types, this can be reduced to a list of 933 minimal analytic gaps, only 163 if counted up to permutation. Finally, we just have to check that each of them satisfies the statement of Theorem~\ref{3break}. Of course, some shortcuts are possible, but this describes the general procedure of finitizing the problem. It was too much work already to try to find the list of minimal analytic 4-gaps, but playing a little bit with types, normal embeddings, etc. one can prove partial results like the following:

\begin{thm}\label{2break}
Let $\Gamma = \{\Gamma_i : i<n\}$ be an $n$-gap. Then, there exists $B\subset n$ of cardinality 2 such that $\Gamma$ can be $B$-broken.
\end{thm}

We have another general result about breaking gaps that follows immediately from the theory explained in Section~\ref{sectionweak}. For every $m<\omega$, let $J(m)$ be the number of types that exists in $m^{<\omega}$. Thus, $J(2)=8$, $J(3)=61$, etc. 

\begin{thm}\label{Jbreak}
For every analytic $n$-gap $\Gamma$ and every set $B\subset n$, there exists a set $B\subset C \subset n$ such that $|C|\leq J(|B|)$ and $\Gamma$ can be $C$-broken.
\end{thm}

The proof would run as follows: Suppose, without loss of generality that $B=m=\{0,\ldots,m-1\}$. Apply Theorem~\ref{criticalgap} to the gap $\{\Gamma_i : i<m\}$, and get (perhaps after a permutation) an injective function $u:m^{<\omega}\To N$ such that $[i]$-chains are sent to elements of $\Gamma_i$. Using Theorem~\ref{weakRamsey}, we can suppose that, for each type $\tau$, either all sets of type $\tau$ are sent by $u$ to some $\Gamma_{f(\tau)}$, or all sets of type $\tau$ are sent to $(\bigcup_{i<n}\Gamma_i)^\perp$. The image of $u$ is then orthogonal to all $\Gamma_i$ which are not of the form $\Gamma_{f(\tau)}$. Since there are only $J(m)$ types in $m^{<\omega}$ we conclude that the image of $u$ is orthogonal to all but $J(m)$ of the $\Gamma_i$'s.\\

The function $J$ is optimal for Theorem~\ref{Jbreak}. A particular instance of this result asserts that if $\Gamma$ is an analytic $n$-gap, and $B\subset n$ is a set of cardinality 2, then there exists $B\subset C \subset n$ with $|C|\leq J(2) = 8$ such that $\Gamma$ can be $C$-broken. Optimality means that we can find an $8$-gap which cannot be $C$-broken for any $C$ with $\{0,1\}\subset C \subsetneq \{0,1\ldots,7\}$. This 8-gap is nothing else than considering $\{\tau_0,\ldots,\tau_7\}$ the 8 types of the dyadic tree, and then define $\Gamma$ so that $\Gamma_i$ is the family of all sets of type $\tau_i$.\\

To finish this section, we would like to make some comments going back to the non-analytic world. We can relativize the notions of jigsaw and clover to a given $B\subset n$:
\begin{itemize}
\item We say that $\Gamma$ is a $B$-clover if it cannot be $B$-broken,
\item We say that $\Gamma$ is a $B$-jigsaw if $\{\Gamma_i|_M : i\in A\}$ can be $B$-broken whenever $A\supset B$, $M\subset N$ and $\{\Gamma_i|_M : i\in A\}$ is a gap,
\end{itemize}

and then we can ask about mixed versions of jigsaws and clovers: Given a family $\mathfrak{X}$ of subsets of $n$, can we find an $n$-gap which is a $B$-jigsaw for $B\in \mathfrak{X}$ and a $B$-clover for $B\not\in\mathfrak{X}$, $|B|\geq 2$? In the analytic case, we know that this is possible only for certain families $\mathfrak{X}$ and this corresponds to the previous discussion. If we ask about arbitrary $n$-gaps, then we were able to give a positive answer to this question \cite[Theorem 26]{multiplegaps} assuming the existence of a completely separable almost disjoint family. We do not know if this can be proved from ZFC alone. Actually it is unknown if the existence of completely separable almost disjoint families follows from ZFC, they are known to exist under certain assumptions like $\mathfrak{c}<\aleph_\omega$ or $\mathfrak{s}<\mathfrak{a}$~\cite{sane}.\\

%
%

\section{Multiple gaps which are countably separated}\label{sectioncountablyseparated}

Some structural theory of gaps is possible in the general non-analytic case when the gaps are countably separated. This is done in \cite[Section 5]{multiplegaps}. Given a topological space $L$ and a subset $a\subset L$, let $acc(a)$ be the set of accumulation points of $a$, that is
$$acc(a) = \bigcap\left\{\overline{a\setminus c} : c\text{ finite}\right\}$$
Fix a countable dense subset $D$ of the Cantor set $L=2^\omega$, and for each subset $G\subset L$, define:
$$\mathcal{I}_G = \{a\subset D : acc(a)\subset G\}$$
We say that the sets $\{G_i : i<n\}$ are separated by open sets if there exist open sets $U_i\supset G_i$ such that $\bigcap_{i<n}U_i = \emptyset$. The $n_\ast$-gaps obtained in this way are \emph{above} any countably separated $n_\ast$-gap:

\begin{thm}
Let $\{G_i : i<n\}$ be subsets of $L$ which cannot be separated by open sets and $\bigcap_{i<n}G_i = \emptyset$. Then $\{\mathcal{I}_{G_i} : i<n\}$ is an $n_\ast$-gap which is countably separated.
\end{thm}

\begin{thm}
Let $\{\Gamma_i : i<n\}$ be a countably separated $n_\ast$-gap on a countable set $N$. Then there exists a bijection $\phi:N\To D$, and sets $\{G_i : i<n\}$ as in the above theorem such that $\phi(\Gamma_i)\subset \mathcal{I}_{G_i}$.
\end{thm}

Say that $\{\Gamma_i : i<n\}$ is strongly countably separated if $\{\Gamma_i : i\in A\}$ is countably separated for each $A\subset n$. The class of \emph{dense strongly countably separated $n$-gaps} is an extreme class of gaps. Among the minimal analytic $n$-gaps, the only ones which are dense and strongly countably separated are the permutations of $\{\Gamma_{M_i} : i<n\}$ where $M_i$ is the set of types $\tau$ with $\max(\tau)=i$. The point is that this is the only class where we know that the structural theory of analytic gaps extends to general non-analytic gaps.

\begin{thm}
If $\Delta = \{\Delta_i : i<n\}$ is a dense strongly countably separated $n$-gap, then there is a permutation $\Gamma^\sigma$ of $\{\Gamma_{M_i} : i<n\}$ such that $\Gamma^\sigma \leq \{\Delta^{\perp\perp}_i : i<n\}$.
\end{thm}

Taking the biorthogonal is a necessary, but not relevant for most applications, restriction in this case. For example, since $\{\Gamma_{M_i} : i<n\}$ can be checked to be a jigsaw, we get, without analyticity assumptions:

\begin{cor}
Every dense strongly countably separated $n$-gap is a jigsaw.
\end{cor}

We notice that the gap $\{\Gamma_{M_i} : i<n\}$ is the same as $\{\mathcal{I}_{G_i} : i<n\}$, where in the above definition of ideals $\mathcal{I}_G$, we consider $L$ to be a countable compact scattered space of height $n+1$ instead of the Cantor set, $D$ is the set of isolated points of $L$, and $G_i = L^{(i+1)}\setminus L^{(i+2)}$ are the levels in the Cantor-Bendixson derivation of $L$.

\section{Multiple gaps and the topology of $\beta\omega\setminus\omega$}\label{sectionomegastar}

The \v{C}ech-Stone compactification $\beta\omega$ is a topological space characterized by the following properties:
\begin{enumerate}
\item $\beta\omega$ contains the set of natural numbers, $\omega\subset\beta\omega$ as a dense subset,
\item each $n\in\omega$ is an isolated point of $\beta\omega$,
\item $\beta\omega$ is a compact space,
\item If $a,b\subset\omega$ and $a\cap b = \emptyset$, then $\overline{a}\cap \overline{b} = \emptyset$.
\end{enumerate}

When we remove isolated points, we obtain the \v{C}ech-Stone remainder $\omega^\ast = \beta\omega\setminus \omega$ of the natural numbers. The point about these spaces is that all problems abouts the family of subsets of $\omega$, endowed with the order $\subset$ can be translated into topological problems on $\beta\omega$, and all problems related to the contention modulo finite $\subset^\ast$ (like those related to gaps) can be translated into topological problems about $\omega^\ast$. This translation procedure is called \emph{Stone duality}, and the reader who is not familiar may be referred to \cite{Walker, Semadeni} as two among many places where one can learn about it. We shall explain now how the gap theory translates through this duality, providing some appealing statements about the space $\omega^\ast$. No detailed argumets will be found in this section, but the fact is that all results are simple excercises of translation through Stone duality, once one gets a basic familiarity with it.\\

Each open subset $U$ of $\omega^\ast$ can be associated to a family $I(U)$ of subsets of $\omega$,
$$I(U) = \{a\subset \omega : acc(a) \subset U\}$$

The families $I(U)$ and $I(V)$ are orthogonal if and only if $U\cap V =\emptyset$. The families $\{I(U_i) : i<\omega\}$ can be separated if and only if $\bigcap_{i<n}\overline{U_i} = \emptyset$. Hence, an $n$-gap corresponds to a finite family of pairwise disjoint open subsets of $\omega^\ast$ whose closures have nonempty intersection.\\

Let us start by translating Theorem~\ref{aleph1}:

\begin{thm}
Assume $MA_{\aleph_1}$, and let $U,V,W$ be three pairwise disjoint open subsets of $\omega^\ast$, each of which is a union of $\aleph_1$ many closed sets. Then $\overline{U}\cap\overline{V}\cap\overline{W} = \emptyset$.
\end{thm}

This constrasts with the fact that, by Hausdorff's construction, there exist two disjoint open sets $U,V$, each a union of $\aleph_1$ many closed sets, such that $\overline{U}\cap\overline{V} \neq\emptyset$.\\

Let us say that an open set $U\subset\omega^\ast$ is analytic if $I(U)$ is analytic. The fact that $\{I(U_i) : i<n\}$ can be $B$-broken is translated topologically into the existence of a point $x$ such that $x\in \bigcap_{i\in B}\overline{U_i}$ but $x\not\in \bigcup_{i\not\in B}\overline{U_i}$. The restriction discussed on Section~\ref{sectionbreak} on the ways that an analytic gap can be broken, translate now into certain forbidden configurations for the partial intersections of closures in a family of analytic open sets. For example, Theorem~\ref{2break} traslates as:

\begin{thm}
If $\{U_i : i<n\}$ are pairwise disjoint analytic open subsets of $\omega^\ast$, then
\begin{enumerate}
\item Either $\{\overline{U_i} : i<n\}$ are pairwise disjoint
\item Or there exists $x\in\omega^\ast$ such that $|\{i<n : x\in\overline{U_i}\}|=2$ 
\end{enumerate}
\end{thm}

Just for fun, let us state the version of Theorem~\ref{Jbreak} for $k=3$. The number $J(3) = 61$ is an optimal bound.

\begin{thm}
If $\{U_i : i<n\}$ are pairwise disjoint analytic open subsets of $\omega^\ast$ such that $\bigcap_{i<3}\overline{U_i} \neq\emptyset$, then there exists $x\in\bigcap_{i<3}\overline{U_i}$ such that $|\{i<n : x\in\overline{U_i}\}|\leq 61$
\end{thm}

\section{Selective coideals and analytic almost disjoint families}\label{sectionalmostdisjoint}

An family $\mathcal{A}$ of infinite subsets of the countable set $N$ is said to be an almost disjoint family if $a\cap b$ is finite for every $a,b\in\mathcal{A}$, $a\neq b$. An example of a closed almost disjoint family is the set $\mathcal{B}$ of all branches of the dyadic tree $2^{<\omega}$. That is, $\mathcal{B} = \{b_x : x\in 2^\omega\}$ where $b_x = \{(x_0,x_1,\ldots,x_{k-1}) : k<\omega\}$ if $x=(x_0,x_1,\ldots)$. The following theorem asserts that every uncountable analytic almost disjoint family contains a copy of the family $\mathcal{B}$ inside.

\begin{thm}\label{almostdisjoint}
If $\mathcal{A}$ is an uncountable analytic almost disjoint family of subsets of a countable set $N$, then there exists two injective functions $u:2^{<\omega}\To N$ and $a:2^\omega\To \mathcal{A}$ such that
\begin{enumerate}
\item $u(b_x) \subset a(x)$ for each $x\in 2^\omega$,
\item $u(c)\in \mathcal{A}^\perp$ whenever $c\in \mathcal{B}^\perp$.
\end{enumerate}
\end{thm}

A more general version of Theorem~\ref{almostdisjoint} can be stated in terms of so-called selective coideals and is provided in \cite{analyticmultigaps}. The proof combines results of Mathias~\cite{Mathias} with the theory explained in Section~\ref{sectionweak}. The idea is to consider the gap $\{\mathcal{A}^\perp,\mathcal{A}\}$, to which we can apply Theorem~\ref{firstdichotomy}. Combining it with Theorem~\ref{weakRamsey}, we find $\{\Gamma_{S_0},\Gamma_{S_1}\}\leq \{\mathcal{A}^\perp,\mathcal{A}\}$ where $S_i$ are sets of types with $[i]\in S_i$. The special structure of the almost disjoint family rules out any top-comb type to belong to $\mathcal{A}$. Composing with the function $w(s_0,s_1,\ldots,s_k) = (1,s_0,1,s_1,\ldots1,s_k)$ one obtains the desired function $u$.

\section{Analytic multiple gaps and sequences in Banach spaces}\label{sectionsequences}

Gaps can be a tool to analyze how different classes of subsequences \emph{get mixed} inside a sequence $(x_n)_{n<\omega}$. We shall illustrate this way of thinking by considering sequences of vectors in normed spaces and certain classes of subsequences. We could, for instance, fix a number $p\in [1,\infty)$ and consider the sequences of vectors which are equivalent to the canonical basis of $\ell_p$, that is:

\begin{dfn}
Let $1\leq p <\infty$. We say that a sequence of vectors $(x_n)_{n<\omega}$ in a normed space is called an $\ell_p$-sequence if there exists $C>0$ such that for every $n_1,\ldots,n_k<\omega$ and every scalars $\lambda_1,\ldots,\lambda_p$
$$ (\clubsuit) \ \frac{1}{C} \cdot \left(\sum_{i=1}^k |\lambda_i|^p \right)^{\frac{1}{p}} \leq \left\|\sum_{i=1}^k \lambda_i x_{n_i}\right\|  \leq C\cdot \left(\sum_{i=1}^k |\lambda_i|^p \right)^{\frac{1}{p}}$$
\end{dfn}

Let us consider now finite sets $\vec{p} \subset [1,+\infty)$

\begin{dfn}
A sequence $(x_n)_{n<\omega}$ of vectors in a normed space is called $\vec{p}$-saturated, if for every subsequence $(x_n)_{n\in A}$ there exists a further subsequence $(x_n)_{n\in B}$, $B\subset A$, which is an $\ell_{p}$-sequence for some $p\in\vec{p}$.
\end{dfn}

If we are given a $\{p_i: i<n\}$-saturated sequence, we can consider the following families of subsets of $\omega$:
$$\Gamma_i = \left\{A\subset \omega : \{x_m\}_{m\in A}\text{ is an }\ell_{p_i}\text{-sequence}\right\}$$
These families are mutually orthogonal, because an infinite sequence cannot be an $\ell_p$-sequence and an $\ell_q$-sequence at the same time, for $p\neq q$. Moreover, the families $\Gamma_i$ are analytic, because $\Gamma_i$ can be written as a countable union of countable intersection of closed families as
$$\bigcup_{C\in\omega\setminus\{0\}}\bigcap_{n_1,\ldots,n_k<\omega}\bigcap_{\lambda_1,\ldots,\lambda_k\in\mathbb{Q}}\left\{A: \{n_1,\ldots,n_k\}\not\subset A, \text{ or }(\clubsuit)\text{ holds for }p=p_i \right\}$$

The fact that $\{\Gamma_i : i<n\}$ form an $n$-gap is equivalent to $(x_n)_{n<\omega}$ being a $\vec{p}$-sequence, according to the following definition:

\begin{dfn}
A sequence $(x_n)_{n<\omega}$ of vectors in a normed space is called a $\vec{p}$-sequence, if it is $\vec{p}$-saturated and we cannot find any finite decomposition $\omega = \bigcup_{i<k}A_i$ so that each $(x_n)_{n\in A_i}$ is $\vec{q}_i$-saturated for some $\vec{q}_i\subsetneq \vec{p}$.
\end{dfn}

It is not really more restrictive to study $\vec{p}$-sequences than $\vec{p}$-saturated sequences, because every $\vec{p}$-saturated sequence can be decomposed into a finite union of subsequences, each of which is a $\vec{q}$-sequence for some $\vec{q}\subset\vec{p}$. We should think of a $\vec{p}$-sequence as a sequence $\{x_m\}_{m<\omega}$ of vectors which contains many subsequences equivalent to $\ell_p$, for $p\in\vec{p}$, and moreover all these subsequences are \emph{well mixed} inside the sequence $\{x_m\}_{m<\omega}$.\\

As we explained above, each $\vec{p}$-sequence gives rise to an analytic $n$-gap $\{\Gamma_i : i<n\}$ which is moreover dense. Therefore, all the theory of analytic $n$-gaps that we have developed can be applied in this context. For each fixed $n$, we know that there are only a finite number of minimal canonical forms in which a $\vec{p}$-sequence can be produced with $|\vec{p}|=n$, that correspond to the minimal analytic $n$-gaps which are dense. The fact that the gap $\{\Gamma_i : i<n\}$ can be $B$-broken is equivalent to saying that the sequence $(x_m)_{m<\omega}$ contains a $\{p_i : i\in B\}$-sequence. Thus, given a $\vec{p}$-sequence, the problem of knowing for which $\vec{q}\subset \vec{p}$ does $(x_m)_{m<\omega}$ contain a $\vec{q}$-sequence, happens to be a nontrivial problem related to the discussion of Section~\ref{sectionbreak}. Just as an example, this is a corollary of Theorem~\ref{2break}:

\begin{thm}
If $|\vec{p}|\geq 2$, then every $\vec{p}$-sequence has a subsequence which is a $\vec{q}$-sequence for some $\vec{q}\subset \vec{p}$ with $|\vec{q}|=2$.
\end{thm}

The tricky thing is that $\vec{q}$ cannot be chosen arbitrarily. We know that there exist dense analytic gaps $\{\Delta_i : i<n\}$ made opf types which cannot be broken for certain $B\subset n$, $|B|=2$. For ecample, if $\Gamma_{\{\tau\}}$ is the family of all subsets of $2^{<\omega}$ of type $\tau$, then $\{\Gamma_{\{\tau\}} : \tau\in\mathfrak{T}_2\}$ cannot be $\{[0],[1]\}$-broken. Out of that, one gets $\vec{q}\subset\vec{p}$ with $|\vec{q}|=2$, for which it is possible to construct a $\vec{p}$-sequence which does not contain any $\vec{q}$-subsequence.

\section{Gaps, selectors, and spaces of continuous functions}\label{sectionCK}

Our original motivation to introduce multiple gaps was some problems concerning operators on the space $C(\omega^\ast)$ of continuous functions on $\omega^\ast$, better known under the name $\ell_\infty/c_0$ among Banach space theorists. We briefly explain these problems in this section. We tried to make it very much self-contained, still we refer to \cite{AlbiacKalton} for some basics on Banach spaces, and to \cite{Semadeni} as a classical reference on Banach spaces of continuous functions.

\begin{dfn}
Let $X\subset Y$ be Banach spaces, and $\lambda\geq 1$. We say that $X$ is $\lambda$-complemented in $Y$ if there exists a linear operator $P:Y\To X$ such that
\begin{enumerate}
\item $P(x) = x$ for all $x\in X$,
\item $\|P(x)\| \leq \lambda \|x\|$ for all $x\in Y$.
\end{enumerate}
\end{dfn}

We are interested only in a particular class of Banach spaces, those of the form $C(K)$. Given a compact space $K$, the space $C(K)$ is the Banach space of real-valued continuous functions on $K$, endowed with the norm $\|f\| = \max\{|f(x)| : x\in K\}$. The assignment $K\mapsto C(K)$ is a contravariant functor, in the sense that each continuous map $f:L\To K$ corresponds to a composition operator $f^0:C(K)\To C(L)$ given by $f^0(\phi) = \phi\circ f$. If $f$ is onto, then $f^0$ is an injective isometry that identifies $C(K)$ as a subspace of $C(L)$.

Suppose that we are given a continuous surjection $f:L\To K$, and we view $C(K)$ as a subspace of $C(L)$ via the operator $f^0$. How can we know, just studying $f$ in topological terms, whether the space $C(K)$ is $\lambda$-complemented in $C(L)$? This situation was analysed by Milutin~\cite{Mil}, and furhter by Pe\l czy\'{n}ski~\cite{Pelcz} and Ditor~\cite{Ditor}. The first observation is that if $f$ has a continuous selector $s:K\To L$, then $C(K)$ is 1-complemented in $C(L)$ by the projection $P(\phi) = \phi\circ s$. The necessary and sufficient condition is that $f$ has a generalized continuous selector which takes measures in $L$ as values, instead of just points. Namely, consider $M_\lambda(L)$ the set of all regular Borel measures on $L$ of variation less than or equal to $\lambda$, endowed with the coarsest topology that makes the map $\mu \mapsto \int_L \phi d\mu$ continuous for each $\phi\in C(L)$. Then,

\begin{prop}\label{averaging}
The space $C(K)$ is $\lambda$-complemented in $C(L)$ if and only if there exists a continuous function $s:K\To M_\lambda(L)$ such that $$\int_L (\phi\circ f) ds(x) = \phi(x)$$ for all $x\in K$ and $\phi\in C(K)$. 
\end{prop}

The projection induced by the generalized selector $s$ is given by $P(\varphi)(x) = \int_L \varphi ds(x)$. The fact that all possible projections are induced by generalized selections follows from the Riesz representation theorem, that states that the dual space $C(L)^\ast$ is identified with the regular Borel measures of finite variation on $L$. We refer to \cite{Semadeni} where all these things are explained in detail.\\

What does all this have to do with gaps? Consider a dense 2-gap $\Gamma = \{\Gamma_0,\Gamma_1\}$ on $\omega$. We already know that it is easy to construct such a thing, we just have to distribute the eight types of the dyadic tree into $\Gamma_0$ and $\Gamma_1$. We can associate two open sets of $\omega^\ast$, $U_i = \bigcup\{acc(a) : a\in \Gamma_i\}$. Because $\{\Gamma_0,\Gamma_1\}$ is a gap, this translates into the fact that $U_0\cap U_1 = \emptyset$ but $\overline{U_0}\cap \overline{U_1} \neq \emptyset$. The density implies that $\overline{U_0}\cup \overline{U_1} = \omega^\ast$. Now let $K=\omega^\ast$, and the \emph{splitted space} $L = \overline{U_0}\times\{0\} \cup \overline{U_1}\times\{1\}$. We have the \emph{gluing} continuous surjection $f:L\To \omega^\ast$, given by $f(x,i) = x$. We are in the situation described above, and using Proposition~\ref{averaging} we can show that $C(\omega^\ast)$ is not 1-complemented in $C(L)$. Namely, suppose that $s:\omega^\ast\To M_1(L)$ was a generalized selection. For $x\in U_i$ we have no other choice than taking $s(x) = \delta_{(x,i)}$, the Dirac measure concentrated in the poing $(x,i)$. But then, taking $y\in\overline{U_0}\cap \overline{U_1}$, the continuity of $s$ would imply that $s(y)$ should be both equal to $\delta_{(y,0)}$ and $\delta_{(y,1)}$, a contradiction. In this way we proved that the existence a dense 2-gap implies:

\begin{prop}
There is a superspace in which $C(\omega^\ast)$ is not 1-complemented. 
\end{prop}

This is just a consequence of Goodner-Nachbin's characterization \cite{Goodner,Nachbin} of 1-injective Banach spaces of continuous functions (those which are 1-complemented in every superspace), cf. \cite{AlbiacKalton}. Indeed the connection between gaps and 1-injectivity is deeper than what is suggested at first sight by the above arguments. For example, the lack of countably generated gaps stated in Proposition~\ref{separationofcountable} has the following consequence, cf.~\cite{extr}:

\begin{prop}
$C(\omega^\ast)$ is 1-complemented in every superspace $Y$ for which $Y/C(\omega^\ast)$ is separable.
\end{prop}

Now, what if we want to get a superspace in which $C(\omega^\ast)$ is not complemented at all (that is, it is not $\lambda$-complemented for any $\lambda\geq 1$). The previous argument to kill all 1-projections does not work for large $\lambda$. For instance, if we are allowed to take a selection $s:K\To M_3(L)$, then we are not anymore obliged to take $s(x) = \delta_{(x,0)}$ for $x\in U_0$. We could take $s(x) = \delta_{(x,0)} - \delta_{(y,0)} + \delta_{(y,1)}$ where $y\in \overline{U}_0\cap\overline{U}_1$. Indeed, if we imagine that $\overline{U}_0\cap\overline{U}_1 = \{y\}$, then this does give a generalized selector, together with $s(x) = \delta_{(x,1)}$ for $x\not\in U_0$. We need to have at least $\lambda\geq 2$ to play this kind of trick. But we may need even larger $\lambda$ if we want to apply the trick twice or more times, if we have many intersections of closures to deal with. Namely, if we have now a dense $n$-gap $\Gamma = \{\Gamma_i : i<n\}$, and associated open sets $\{U_i : i<n\}$, we can consider again the splitted space $L = \bigcup_{i<n}\overline{U_i}\times\{i\}$ and the continuous gluing surjection $f:L\To\omega^\ast$ given by $f(x,i)=x$, and we have that

\begin{prop}\label{jigsawditor}
If $\Gamma$ is a dense $n$-jigsaw, and $n=2^m$, then $C(\omega^\ast)$ is not $m$-complemented in $C(L)$.
\end{prop}

Being a jigsaw is translated topologically into the fact that every point $x\in\bigcap_{i\in A}\overline{U_i}\setminus\bigcap_{i\not\in A}\overline{U_i}$ belongs to the closure of $\bigcap_{i\in B}\overline{U_i}\setminus\bigcap_{i\not\in B}\overline{U_i}$ whenever $B\subset A$. These are the \emph{many intersections of closures} mentioned above, that require more and more norm on the measures to be able to get a generalized continuous selector. We will not explain the details of this here, we refer to \cite[Section 8]{multiplegaps}, and the results of \cite{Ditor} cited there. We have already mentioned in Section~\ref{sectionbreak} that there are natural examples of analytic dense jigsaws to provide inputs for Proposition~\ref{jigsawditor}. It is an easy excercise (look at the last paragraph of the paper \cite{multiplegaps}) to glue together all the $C(L)$'s given by Proposition~\ref{jigsawditor} into a single one:

\begin{thm}
There is a superspace $C(\omega^\ast)\subset C(L)$ in which $C(\omega^\ast)$ is not $\lambda$-complemented for any $\lambda$. 
\end{thm}

This result was originally proved by Amir~\cite{Amir}. Both Amir's and our approach allow to obtain a space $C(L)$ of the cardinality of the continuum. Our original (unsuccessful) intention when introducing multiple gaps was trying to apply these ideas to deal with the following problem posed in \cite{extr}: Is there a superspace $Y\supset C(\omega^\ast)$ in which $C(\omega^\ast)$ is not complemented and $Y/C(\omega^\ast)$ has density $\aleph_1$? Of course, we mean if this is provable in ZFC alone. If we say 1-complemented instead of complemented then the answer is: yes, and the space $Y$ can be obtained out of a Haudorff's gap like in Theorem~\ref{Hausdorff}. However, we cannot hope multiple gaps to provide examples of that kind, because of Theorem~\ref{aleph1}.

\end{document}